\documentclass[preprint,authoryear,12pt]{elsarticle}

\usepackage{amsmath,amsopn,amstext,amscd,amsfonts,amssymb}
\usepackage{natbib}
\usepackage{graphicx}
\usepackage{verbatim,color}
\usepackage{animate}

\def\tr{\mathop{\rm tr}\nolimits}
\def\re{\mathop{\rm Re}\nolimits}
\def\diag{\mathop{\rm diag}\nolimits}

\def\vec{\mathop{\rm vec}\nolimits}

\def \Vol {\mathop{\rm Vol}\nolimits}
\def\etr{\mathop{\rm etr}\nolimits}
\def\diag{\mathop{\rm diag}\nolimits}

\newcommand {\boldgreektext}[1] {\boldmath
             \(#1\)\unboldmath}
\newcommand {\boldgreek}[1]
             {\mbox{\boldgreektext{#1}}
            }

\newtheorem{theorem}{Theorem}

\newtheorem{lemma}{Lemma}

\newdefinition{definition}{Definition}
\newdefinition{remark}{Remark}
\newproof{proof}{Proof}

\begin{document}
\begin{frontmatter}
\title{Probabilities in multimatrix variate distributions: an application in SARS-CoV-2}

\author[fjc]{Francisco J. Caro-Lopera\corref{cor1}}
\ead{fjcaro@udemedellin.edu.co}

\author[jadg]{Jos\'e A. D\'{\i}az-Garc\'{\i}a}
\ead{jadiaz@cimat.mx}

\cortext[cor1]{Corresponding author}

\address[fjc]{University of Medellin, Faculty of
Basic Sciences, Carrera 87 No.30-65, Medellin, Colombia}

\address[jadg]{Universidad Aut\'onoma de Chihuahua, Facultad de Zootecnia y Ecolog\'{\i}a,
Perif\'erico Francisco R. Almada Km 1, Zootecnia, 33820 Chihuahua, Chihuahua, M\'exico.}

\begin{abstract}
Recently the termed \emph{multimatrix variate distributions} were proposed in \citet{dgcl:24a} as an alternative for univariate and vector variate copulas. The distributions are based on sample probabilistic dependent elliptically countered models and most of them are also invariant under this family of laws. Despite a large of results on matrix variate distributions since the last 70 years, the spherical multimatrix distributions and the associated probabilities on hyper cones can be computable. The multiple probabilities are set in terms of recurrent integrations allowing several matrix computation a feasible task. An application of the emerging probabilities is placed into a dynamic molecular docking in the SARS-CoV-2 main protease. Finally, integration over multimatrix Wishart distribution provides a simplification of a complex kernel integral in elliptical models under real normed division algebras and the solution was applied in elliptical affine shape theory.    
\end{abstract}
\end{frontmatter}

\section{Introduction}\label{sec:1}
The theory of matrix variate distributions has been developed profusely in the last century. It has emerged as a natural extension of the univariate case, covering all statistical and probabilistic disciplines.

However, the new extensions are becoming more complicated, giving rise to complex calculations that remained impossible for decades. Other generalizations took longer, specifically those related to joint distributions of dependent samples in non-Gaussian models and had usually been addressed by theories such as copulas among others.

Each new extension of the univariate or vector case brings with it not only computability, but increasingly demanding applications in support distributions, estimation and dependence. In response to such goal, alternatives to copula theory have appeared, which not only extend the results of the univariate and vector cases, but also consider all real normed division algebras (real, complex, quaternion, octonion).

Depending on the type of function that describes the relationships of the multiple matrices, two types of distributions have recently been proposed applied to probabilistically dependent samples and under robust elliptically countered  models. These are called multimatricvariate and multimatrix variate distributions and the seminal works appear in \citet{dgcl:22}, \citet{dgcl:24a} and \citet{dgcl:24b}. The multivector variate case can be seen in \citet{dgclpr:22}. The importance of dependent samples, robust distributions beyond Gaussian, general calculations for all real normed division algebras and estimation based on non-independent likelihoods have been studied in detail in the aforementioned articles.

In this work we focus on the computational problem of distributions and on the calculation of multimatrix probabilities.

A multimatrix theory that solves the addressed previous problems but that cannot be computed would add to the hundreds of theoretical results in matrix analysis developed in the last seven decades. Since that time, the advent of the theory of integration on orthogonal groups due to A.T. James, allowed the calculation of the joint distribution of the latent roots of any function of a positive definite matrix, among many others applications, see for example \citet{j:60}. The solution created a set of orthogonal polynomials for the expansion of the power of the trace of a matrix and with it the central matrix variate statistics on Gaussian or elliptic models could be built in all real normed division algebras. Only a few years ago the results involving hypergeometric type series of zonal polynomials could be approximated numerically, so it is quite a challenge for new theories that can be expressed in computable distributions. The central and isotropic cases have such an opportunity, but the non-central case, useful for solving, for example, the distribution of the latent roots of a Wishart matrix, came about with the creation in 1979 of an extension of the James zonal polynomials and which A.W.Davis called invariant polynomials of several matrix arguments (\citet{d:79}). The success of the calculation of zonal polynomials lay in their recurring construction using the Laplace-Beltrami operator (\citet{j:68}), so Davis conjectured in 1979 that his polynomials could have such a construction. Unfortunately, in 2016 the impossibility of the method was demonstrated, leaving open the calculation problem for results involving Davis polynomials (\citet{cl:16}). In this context, and for now, only results involving at most zonal polynomials can be computed and be useful for applications. A natural challenge for a multimatrix theory is that the joint distribution can be integrated to compute matrix event probability problems. This aspect has had little treatment in the literature and very few explicit calculations of probabilities in cones are known. The problem of calculating multiple probabilities in a joint distribution is precisely what we also hope to solve, in such a way that at most they are expressed in computable series of zonal polynomials.

The above discussion is placed in the present article by computing some probabilities in multimatrix variate distributions. Some preliminary integrals of interest as well as the definition of the
matrix variate elliptical contoured distributions and generalized series of zonal polynomials are collected in Section \ref{sec:2}.  Then, Section \ref{sec:3} revises the problem of computation in matrix variate distribution theory and provides some results on multiple probabilities on cones written in terms of computable series of zonal polynomials. Some of the probabilities are invariant for the complete family of elliptical distributions, meanwhile the non invariant distributions promoted the simplification of kernel elliptical integrals indexed by general derivatives; that result is applied in affine shape theory under real normed division algebras. Section \ref{sec:4} is based on the dependent joint distribution of a sample derived from a molecular docking in a new cavity of SARS-CoV-2 main protease. Then a simulation of a dynamic molecular docking is set in terms of the probabilities arising from a deformation of the ligand into the corresponding binding site.

\section{Preliminary results}\label{sec:2}

Matrix notations, matrix variate elliptical contoured distributions, zonal polynomials and some multimatrix variate distributions are presented in this section.
First, we start with notations and terminologies. $\mathbf{A}\in \Re^{n \times m}$ denotes a \emph{matrix} with $n$
rows and $m$ columns; $\mathbf{A}'\in \Re^{m \times n}$ is the \emph{transpose matrix}, and if
$\mathbf{A}\in \Re^{n \times n}$ has an \emph{inverse}, it will be denoted by $\mathbf{A}^{-1} \in
\Re^{n \times n}$. $\mathbf{A}\in \Re^{n \times n}$ is a \emph{symmetric matrix} if $\mathbf{A} =
\mathbf{A}'$. If all their eigenvalues are positive then $\mathbf{A}$ is a \emph{positive definite
matrix},  a fact denoted as $\mathbf{A} > \mathbf{0}$. An \emph{identity matrix} will be denoted by
$\mathbf{I}\in \Re^{n \times n}$. To specify the size of the identity, we will use $\mathbf{I}_{n}$. $\tr
(\mathbf{A})$ denotes the trace of matrix $\mathbf{A} \in \Re^{m \times m}$. If $\mathbf{A}\in \Re^{n
\times m}$ then by $\vec (\mathbf{A})$ we mean the $mn \times 1$ vector formed by stacking the columns of
$\mathbf{A}$ under each other; that is, if $\mathbf{A} = [\mathbf{a}_{1}\mathbf{a}_{2}\dots
\mathbf{a}_{m}]$, where $\mathbf{a}_{j} \in \Re^{n \times 1}$ for $j = 1, 2, \dots,m$, 
$
  \vec(\mathbf{A})= \left [
                       \mathbf{a}_{1}' \\
                       \mathbf{a}_{2}' \\
                       \cdots \\
                       \mathbf{a}_{m}'
                    \right ]'.
$
The \textit{Frobenius norm} of a matrix $\mathbf{A}$ will be denoted as $||\mathbf{A}||$. Typically the
Frobenius norm is denoted as $||\mathbf{A}||_{F}$, to differentiate it from other matrix norms. Since
we will use only the Frobenius norm, it just be denoted as $||\mathbf{A}||$. It is
defined by
$$
  ||\mathbf{A}|| = \sqrt{\tr(\mathbf{A}'\mathbf{A})}=\sqrt{\vec'(\mathbf{A})\vec(\mathbf{A})}.
$$

Finally, $\mathcal{V}_{m,n}$ denotes the \emph{Stiefel manifold}, the space of all matrices
$\mathbf{H}_{1} \in \Re^{n \times m}$ ($n \geq m$) with orthogonal columns, that is, $\mathcal{V}_{m,n} =
\{\mathbf{H}_{1} \in \Re^{n \times m}; \mathbf{H}'_{1}\mathbf{H}_{1} = \mathbf{I}_{m}\}$. In addition, if
$(\mathbf{H}'_{1}d\mathbf{H}_{1})$ defines an \textit{invariant measure on the Stiefel manifold}
$\mathcal{V}_{m,n}$, from Theorem 2.1.15, p. 70 in \cite{mh:05},
\begin{equation}\label{evs}
  \int_{\mathcal{V}_{m,n}} (\mathbf{H}'_{1}d\mathbf{H}_{1}) = \frac{2^{m} \pi^{mn/2}}{\Gamma_{m}[n/2]}.
\end{equation}
where $\Gamma_{m}[a]$ denotes the multivariate Gamma function, see \citet[Definiton 2.1.10, p. 61]{mh:05}

Now, let $\mathbf{V} \in \Re^{N \times m}$ random matrix with a \textit{matrix variate elliptical
distribution} with respect to the Lebesgue measure $(d\mathbf{V})$, see \citet{gv:93}. Therefore its density function is
given by
\begin{equation}\label{elliptical}
 dF_{_{\mathbf{V}}} (\mathbf{V}) = |\mathbf{\Sigma}|^{-N/2}|\mathbf{\Theta}|^{-m/2}
  h\left\{\tr\left[(\mathbf{V}-\boldsymbol{\mu})^{T}\mathbf{\Sigma}^{-1}(\mathbf{V}-
  \boldsymbol{\mu})\mathbf{\Theta}^{-1}\right]\right\}(d\mathbf{V}).
\end{equation}
The location  parameter is $\boldsymbol{\mu}\in \Re^{N \times m}$; and the scale parameters
$\mathbf{\Sigma}\in \Re^{N \times N}$ and $\mathbf{\Theta}\in \Re^{m \times m}$, are positive definite
matrices. The distribution shall be denoted by $\mathbf{V}\sim \mathcal{E}_{N \times m}(\boldsymbol{\mu}
,\mathbf{\Sigma}, \mathbf{\Theta}; h)$, and indexed by the kernel function $h\mbox{: } \Re \to
[0,\infty)$, where $\int_{0}^\infty u^{Nm/2-1}h(u)du < \infty$.

When $\boldsymbol{\mu} = \mathbf{0}$, $\mathbf{\Sigma} = \mathbf{I}_{N}$ and $\mathbf{\Theta} =
\mathbf{I}_{m}$ a special case of a matrix variate elliptical distribution appears, in this case it is
said that $\mathbf{V}$ has a \textit{matrix variate spherical distribution}.

\texttt{}Note that for constant $a \in \Re$, the substitution $v = u/a$ and $(du) =
a(dv)$ in \citet[Equation 2.21, p. 26]{fzn:90} provides that
\begin{equation}\label{int}
    \int_{v > 0} v^{Nm/2-1}h(a v) (dv)= \frac{a^{-Nm/2} \Gamma_{1}[Nm/2]}{\pi^{Nm/2}}
\end{equation}

For some applications we will require James zonal polynomials, notated by $C_{\rho}(\mathbf{A})$, see \citet{mh:05}. They are an orthogonal base expansion in $(\tr \mathbf{A})^{r}=\sum_{\rho \in r}C_{\rho}(\mathbf{A})$. Here  $\mathbf{A}\in \Re^{m \times m}$ is positive definite and the summation is over all ordered partitions $\rho=(r_{1},...,r_{m})$, into not more than $m$ parts, such that $r_{1}\geq r_{2}\geq \cdots \geq r_{m}\geq 0$. 
Series of zonal polynomials in this work will required the following expression given by \citet{cldggf:10} and \citet{cldg:12}:
\begin{equation}\label{genhf}
    {}_{p}^{r}P_{q}[f(r,\mathbf{X}):a_{1},\ldots,a_{p};b_{1},\ldots,b_{q};\mathbf{X}]=\sum_{r=0}^{\infty}\frac{f(r,\mathbf{X})}{r!}\sum_{\rho\in r}\frac{(a_{1})_{\rho}\cdots(a_{p})_{\rho}}{(b_{1})_{\rho}\cdots(b_{q})_{\rho}}C_{\rho}(\mathbf{X})
\end{equation} 
where the function $f(r,\mathbf{X})$ is independent of $\rho$ and $(w)_{\rho}=\prod_{i=1}^{m}(w-\frac{1}{2}(i-1))_{r_{i}}$, with $(w)_{0}=1$, $(w)_{t}=w(w+1)\cdots(w+t-1)$ and the $a_{i}'s$ and $b_{j}'s$ are complex numbers. If $f(r,\mathbf{X})=1$, then the well known hypergeometric series ${}_{p}F_{q}(a_{1},\ldots,a_{p};b_{1},\ldots,b_{q};\mathbf{X})$ is obtained, see for example \citet{mh:05}.
We also introduce the notation 
\begin{equation}\label{genpartsum}
    {}_{p}^{r}Q_{q}[a_{1},\ldots,a_{p};b_{1},\ldots,b_{q};\mathbf{X}]=\sum_{\rho\in r}\frac{(a_{1})_{\rho}\cdots(a_{p})_{\rho}}{(b_{1})_{\rho}\cdots(b_{q})_{\rho}}C_{\rho}(\mathbf{X})
\end{equation} 

Finally, we focus on some multimatrix variate distributions derived by \citet{dgcl:24a} which are invariant under family of elliptically countered distributions. i.e. they are not dependent of the kernel function; a preferable property for applications that avoids any prior knowledge of the underlying distribution.

\begin{lemma}\label{mggb2}
Assume that $\mathbf{X} = \left(\mathbf{X}'_{0}, \dots,\mathbf{X}'_{k} \right)'$ has a matrix variate
spherical distribution, with $\mathbf{X}_{i} \in \Re^{n_{i} \times m}$, $n_{i} \geq m$, $i = 0,1, \dots,
k$. Define $V = ||\mathbf{X}_{0}||^{2}$ and $\mathbf{T}_{i} = V ^{-1/2}\mathbf{X}_{i}$, $i = 1,\dots,k$.
The termed  \emph{multimatrix variate Pearson type VII} is the marginal density
$dF_{\mathbf{T}_{1}, \dots,\mathbf{T}_{k}}(\mathbf{T}_{1}, \dots,\mathbf{T}_{k})$ of $\mathbf{T}_{1},
\dots,\mathbf{T}_{k}$ and is given by
\begin{equation}\label{mp7}
   \frac{\Gamma_{1}[Nm/2]}{\pi^{(N-n_{0})m/2}\Gamma_{1}[n_{0}m/2]}
  \left(1+\displaystyle\sum_{i=1}^{k}||\mathbf{T}_{i}||^{2}\right)^{-Nm/2}
  \bigwedge_{i=1}^{k}\left(d\mathbf{T}_{i}\right),
\end{equation}
where $\mathbf{T}_{i} \in \Re^{n_{i} \times m}$, $n_{i} \geq m$ and $N = n_{0}+n_{1}+\cdots+n_{k}$.
Now define $\mathbf{F}_{i} =
\mathbf{T}'_{i}\mathbf{T}_{i} > \mathbf{0}$, $i = 1, \dots,k$. 
The \emph{multimatrix variate beta type II distribution}
 $dF_{\mathbf{F}_{1}, \dots, \mathbf{F}_{k}}(\mathbf{F}_{1}, \dots,\mathbf{F}_{k})$
is given by
\begin{equation}\label{mb2}
     \frac{\Gamma_{1}[Nm/2]}{\Gamma_{1}[n_{0}m/2]\displaystyle\prod_{i=1}^{k}\Gamma_{m}[n_{i}/2]}
  \prod_{i=1}^{k}|\mathbf{F}_{i}|^{(n_{i}-m-1)/2}
  \left(1+\displaystyle\sum_{i=1}^{k}\tr\mathbf{F}_{i}\right)^{-Nm/2}
  \bigwedge_{i=1}^{k}\left(d\mathbf{F}_{i}\right).
\end{equation}
\end{lemma}

\begin{lemma}\label{ggbI}
Suppose that $\mathbf{X} = \left(\mathbf{X}'_{0}, \dots, \mathbf{X}'_{k} \right)'$ has a matrix variate
spherical distribution, with $\mathbf{X}_{i} \in \Re^{n_{i} \times m}$, $n_{i} \geq m$, $i = 0,1, \dots,
k$. Define $V = ||\mathbf{X}_{0}||^{2}$ and $\mathbf{R}_{i} = \left(V +
||\mathbf{X}_{i}||^{2}\right)^{-1/2} \mathbf{X}_{i}$, $i = 1,\dots,k$. The \emph{multimatrix variate Pearson type II distribution} $dF_{\mathbf{R}_{1}, \dots,\mathbf{R}_{k}}(\mathbf{R}_{1},
\dots,\mathbf{R}_{k})$ is
$$
  \frac{\Gamma_{1}[Nm/2]}{\pi^{(N-n_{0})m/2}\Gamma_{1}[n_{0}m/2]} \left[ 1+\displaystyle\sum_{i=1}^{k}
  \frac{||\mathbf{R}_{i}||^{2}}{\left(1- ||\mathbf{R}_{i}||^{2}\right)}\right ]^{-Nm/2} \hspace{2cm}
$$
\begin{equation}\label{mp2}
  \hspace{3cm}\times  \prod_{i=i}^{k}\left(1-
  ||\mathbf{R}_{i}||^{2}\right)^{-n_{i}m/2-1}
  \bigwedge_{i=1}^{k}\left(d\mathbf{R}_{i}\right),
\end{equation}
Moreover, Assuming that $\mathbf{B}_{i} = \mathbf{R}'_{i}\mathbf{R}_{i} > \mathbf{0}$ and $\tr \mathbf{B}_{i} \leq 1$ with $i = 1, \dots,k$, the \emph{multimatrix variate beta type I distribution} $dF_{\mathbf{B}_{1}, \dots, \mathbf{B}_{k}}(\mathbf{B}_{1},
\dots,\mathbf{B}_{k})$ is written as
$$
  \frac{\Gamma_{1}[Nm/2]}{\Gamma_{1}[n_{0}m/2]} \prod_{i=1}^{k}\left(\frac{|\mathbf{B}_{i}|^{(n_{i}-m-1)/2}}{\Gamma_{m}[n_{i}/2]}\right )
  \left[ 1+\displaystyle\sum_{i=1}^{k} \frac{\tr\mathbf{B}_{i}}{\left(1-
  \tr\mathbf{B}_{i}\right)}\right ]^{-Nm/2}\hspace{3cm}
$$
\begin{equation}\label{b1}
  \hspace{4cm}
  \times  \prod_{i=i}^{k}\left(1-
  \tr \mathbf{B}_{i}\right)^{-n_{i}m/2-1}
  \bigwedge_{i=1}^{k}\left(d\mathbf{B}_{i}\right).
\end{equation}
\end{lemma}

Finally

\begin{lemma}\label{ggW} Assume that $\mathbf{X} = \left(\mathbf{X}'_{0}, \dots,
\mathbf{X}'_{k} \right)'$ has  a matrix variate spherical distribution, with $\mathbf{X}_{i} \in
\Re^{n_{i} \times m}$, $n_{i} \geq m$, $i = 0,1, \dots, k$. Define $V = ||\mathbf{X}_{0}||^{2}$ and
$\mathbf{W}_{i} = \mathbf{X}'_{i}\mathbf{X}_{i} > \mathbf{0}$, $i = 1, \dots,k$. \newline Then, the joint density
$dF_{V,\mathbf{W}_{1}, \dots,\mathbf{W}_{k}}(v,\mathbf{X}_{1}, \dots,\mathbf{W}_{k})$ is given by
\begin{equation}\label{mggw}
   \frac{\pi^{Nm/2} v^{Nm/2-1}}{\Gamma_{1}[n_{0}m/2]}\prod_{i=1}^{k}\left(\frac{|\mathbf{W}_{i}|^{(n_{i}-m-1)/2}}{\Gamma_{m}[n_{i}/2]}\right)
   h\left[v+\displaystyle\sum_{i=1}^{k}\tr \mathbf{W}_{i}\right]
   (dv)\bigwedge_{i=1}^{k}\left(d\mathbf{W}_{i}\right),
\end{equation}
where $V > 0$. This distribution shall be termed \emph{multimatrix variate generalised Gamma - generalised Wishart
distribution}.
The joint density function
$dF_{\mathbf{W}_{1}, \dots,\mathbf{W}_{k}}(\mathbf{X}_{1}, \dots,\mathbf{W}_{k})$ is not invariant under spherical elliptical functions but it can be written in following closed form:
\begin{equation}\label{mggwv}
   \pi^{(N-n_{0})m/2} \prod_{i=1}^{k}\left(\frac{|\mathbf{W}_{i}|^{(n_{i}-m-1)/2}}{\Gamma_{m}[n_{i}/2]}\right)
   h\left[\displaystyle\sum_{i=1}^{k}\tr \mathbf{W}_{i}\right]
   \bigwedge_{i=1}^{k}\left(d\mathbf{W}_{i}\right),
\end{equation}
This marginal distribution shall be termed \emph{multimatrix variate generalised Wishart
distribution}.
\end{lemma}
As in the preceding lemmas, result (\ref{mggw}) was derived in \citet{dgcl:24a}. However, proposing the joint marginal (\ref{mggwv}) by expansion of $h\left[v+\displaystyle\sum_{i=1}^{k}\tr \mathbf{W}_{i}\right]$ will lead a simplification of complex kernel integrals involving general derivatives of elliptical models.

Observe that the parameter domain of the multimatrix variate distributions can be extended to the complex or real fields. However their geometrical and/or statistical explication perhaps can be lost. These distributions are valid if we replace $n_{i}/2$ by  $a_{i}$, $n_{0}m/2$ by $a_{0}$ and $Nm/2$ by $a$. Where the $a'^{s}$ are complex numbers with positive real part. From a practical point of view for parametric estimation, this domain extension allows the use of nonlinear optimisation rather integer nonlinear optimisation, among other possibilities.

\section{Computation and Probabilities on cones}\label{sec:3}
We devote a few lines to the problem of computation of certain distributions of type (\ref{genhf}).  Certainly, there are hundreds of papers about matrix variate distribution theory, however the explicit computability of similar results to (\ref{genhf}) is not usually addressed or applied. For example, the excellent work of \citet{mh:05} has attempted certain exact distributions by some approximations, but the majority of the results in non-central models forces the appearance and computation of A.T. James zonal polynomials of one matrix argument (\citet{j:64}) or A.W. Davis invariant polynomials of several matrix arguments (\citet{d:80}). The last problem was closed recently in \citet{cl:16} by refuting \citet{d:79} conjecture about a recursion computation of Davis functions in the same way as James polynomials, leaving intractable a number of theoretical results in non-central distribution theory. But, the remarkable property of James polynomials (\citet{j:68}) (also known Jack polynomials in real normed division algebras) opened the possibility of computing a number of old series of zonal polynomials. The classical central cases on positive definite matrices were expressed in terms of hypergeometric series; for example probabilities of Wisharts matrices bounded by positive definite matrices, and so on. The key fact consists of provide a numerical computation of $(\tr \mathbf{A})^{r}=\sum_{\rho \in r}C_{\rho}(\mathbf{A})$, where  $\mathbf{A}\in \Re^{m \times m}$ is positive definite, see details in (\ref{genhf}). The numerical solution appeared very late in \citet{ke:06} after more than a half century of theoretical results. Given the addressed James recurrence construction proposed in 1968, those algorithms are sufficient for low values of $m$, because the partitions $\rho$ are trivially truncated by $m$ parts, but exact expressions for $C_{\rho}(\mathbf{A})$ with arbitrary $m=r$ are out of any knowledge. In that context, and referring to S. Ramanujan and G.H. Hardy, knowing an exact formula for the number of partitions $\rho$ of arbitrary $r$ is one of the biggest problem of mathematics in all history. The problem is such challenging that is so far to be included into the reasonable famous list of the Millennium Prize Problems (\citet{cmic:00}). If the number of $C_{\rho}(\mathbf{A})$ in the expansion of $(\tr \mathbf{A})^{r}$ is just the number of partitions, just imagine the problem of providing an exact formula for $\sum_{\rho \in r}C_{\rho}(\mathbf{A})$. Other apparitions of the number of partitions arrive in permanents (\citet{clgfbn:13}), and the general derivatives of a composite function and a Kotz model generator (\citet{cldggf:10}), among many others. Moreover, computing generalized hypergeometric series of zonal polynomials (\ref{genhf}) only can be achieved under truncation, low values of $m$ and suitable functions depending on the series index. A number of such series were computed recently by modification of the algorithms of \citet{ke:06} in the context of statistical shape theory, see \citet{cldg:12} and the references therein. In other context, several families of computable polynomial distributions based on zonal functions can be also seen in \citet{cl:18}.

Under this point of revision, we can mention that all the multimatricvariate and multimatrix variate distributions presented in \citet{dgcl:22}, \citet{dgcl:24a} and \citet{dgcl:24b} are completelly computable. Most of them are free of series representation, then the computation is straighforwardly. Such is the case of multimatrix variate Pearson and Beta type distributions given in (\ref{mp7}), (\ref{mp2}), (\ref{mb2}) and (\ref{b1}). The highlighted distributions involve the important property of invariance under the spherical family with generator $h(\cdot)$; this is crucial for a researcher, because no previous knowledge of the underlying distribution is required. No fitting distribution test must be done, except that the general assumption on ellipticity should be held.

Now, if the computation of matrix variate distribution is problematic, we can imagine the issues involved in finding a matrix probability. It should be noted that even simple univariate pdfs and cdfs for small, large or any beta, F and Wishart latent roots themselves have involved important historical problems in the last century. Discarding the trivial probabilities emerging from integrals involving zonal polynomials of an specific partition, it seems that the only existing computable matrix bounded probability are those for central Wishart ($\mathbf{W}$) and beta ($\mathbf{B}$) distribution. The Wishart lower bound probability $P(\mathbf{W}<\mathbf{\Omega})$ and the integral for matrix beta function back to \citet{c:63} (See \citet{acl:21} for the $P(\mathbf{B}<\mathbf{\Omega})$). However, as \citet{c:63} states: ``the complementary probabilities ($P(\mathbf{W}>\mathbf{\Omega})$, and $P(\mathbf{B}<\mathbf{\Omega})$) seem difficult to evaluate". The Beta probability appeared recently in \citet{acl:21}. For the Wishart probability, the explicit solution in the Gaussian case arrived in 1982 (\citet{mh:05}),  and a revision of the underlying proof was given in \citet{clgfbn:16} joint with a generalization to elliptically contoured distributions for both lower and upper probabilities indexed by kernel ($h(\cdot)$). However, probabilities for rectangular matrices and non symmetric square matrices are still open problems.              
In fact, excepting for the interest of small and large latent roots, applications of the existing probabilities on Gaussian or elliptical Wishart distributions and classical Beta matrices seem out of consideration in literature. We try to set here the interest for those probabilities, in particular emerging from a multimatrix context, which naturally arises in experiment of probabilistic dependent samples. In this aspect, once the multimatrix joint distributions involve simplicity, we expect that some related measures considers tractable series representation. A suitable next stage of application in multimatrix variate Pearson and Beta type distributions should consider the computation of probabilities. In particular, \emph{multimatrix variate beta type II distribution} in (\ref{mb2}) has a potential application that we will explore in Section \ref{sec:4}. For example,  the following result considers the probability that $\mathbf{A}_{i}-\mathbf{F}_{i}$ is a positive definite matrix for all constant positive definite matrices $\mathbf{A}_{i}$, $i=1,\ldots,k$.

\begin{theorem}\label{probbetaII} Assume the hypothesis of Lemma \ref{mggb2} and suitable  parameters for convergence of series (\ref{genhf}). For positive definite matrices $\mathbf{A}_{i} \in \Re^{m \times m}$, $i=1,\ldots,k$, the probability $P(\mathbf{0}<\mathbf{F}_{1}<\mathbf{A}_{1},\ldots,\mathbf{0}<\mathbf{F}_{k}<\mathbf{A}_{k},)$ is given by:
$$
\frac{\Gamma_{1}[Nm/2]\Gamma_{m}^{k}[(m+1)/2]}{\Gamma_{1}[n_{0}m/2]\displaystyle\prod_{i=1}^{k}\Gamma_{m}[(n_{i}+m+1)/2]}
  \prod_{i=1}^{k}|\mathbf{A}_{i}|^{n_{i}/2}\times S(r_{2},\ldots,r_{k},\mathbf{A}_{1},\ldots,\mathbf{A}_{k})\hspace{7cm},
$$
where the nested summation $S(r_{2},\ldots,r_{k},\mathbf{A}_{1},\ldots,\mathbf{A}_{k})$ is given 
by
$$
{}_{1}^{r_{k}}P_{1}\left[h_{k}{}_{1}^{r_{k-1}}P_{1}\left[h_{k-1}{}_{1}^{r_{k-2}}P_{1}{}_{1}^{r_{2}}P_{1}\left[h_{2}{}_{1}^{r_{1}}P_{1}\left[h_{1}:a_{1};b_{1};-\mathbf{A}_{1}\right]:a_{2};b_{2}; -\mathbf{A}_{2}\right]\cdots\right]
\right.:$$
\begin{equation}\label{probbetaIIeq}
\begin{tiny}
    :\left.\left.a_{k-1};b_{k-1}; -\mathbf{A}_{k-1}\right]:a_{k};b_{k}; -\mathbf{A}_{k}\right].
\end{tiny}
\end{equation}
Here (\ref{probbetaIIeq}) involves $k$ sums of type (\ref{genhf}) and depends on the indexes $r_{1},\ldots,r_{k}$ and matrices $\mathbf{A}_{1},\ldots,\mathbf{A}_{k}$,
where $h_{k}=(c)_{r_{k}}$,$h_{j}=\left(c+\sum_{i=j+1}^{k}r_{i}\right)_{r_{j}},j=1,\ldots,k-1$, $c=Nm/2$, $a_{i}=n_{i}/2,b_{i}=(n_{i}+m+1)/2, i=1,\ldots,k$.
\end{theorem}
\begin{proof}
Let $c=Nm/2$, 
$$
  M = \frac{\Gamma_{1}[Nm/2]}{\Gamma_{1}[n_{0}m/2]\prod_{i=1}^{k}\Gamma_{m}[n_{i}/2]}, 
$$
$a_{i}=n_{i}/2, g_{i}=(n_{i}-m-1)/2$, and $b_{i}=(n_{i}+m+1)/2$, $i=1,\dots,k$. Consider backwards integration starting in $\mathbf{0}<\mathbf{F}_{k}<\mathbf{A}_{k}$. To avoid the appearance of generalized binomial expansions of zonal polynomials of sums of matrices, the integrand is written as
$$
  \prod_{j=1}^{k-1}|\mathbf{F}_{j}|^{g_{j}}|\mathbf{F}_{k}|^{g_{k}}\left(1+\displaystyle\sum_{j=1}^{k-1}\tr\mathbf{F}_{j} + \tr\mathbf{F}_{k}\right)^{-c}.
$$ 
Performing a convergent power expansion of the form $(w+x)^{-p}=\sum_{i=0}^{\infty}\frac{(p)_{i}w^{-p-i}}{i!}$ and a variable substitution $\mathbf{F}_{k}=\mathbf{A}_{k}^{1/2}\mathbf{Z}_{k}\mathbf{A}_{k}^{1/2}$, with Jacobian $(d\mathbf{F}_{k})=|\mathbf{A}_{k}|^{(m+1)/2}(d\mathbf{Z}_{k})$, lead to the application of \citet[Th. 7.2.10]{mh:05} and (\ref{genhf}).  
Thus,
$$\prod_{j=1}^{k-1}|\mathbf{F}_{j}|^{g_{j}}\int_{\mathbf{0}<\mathbf{F}_{k}<\mathbf{A}_{k}}|\mathbf{F}_{k}|^{g_{k}}\left(1+\displaystyle\sum_{j=1}^{k}\tr\mathbf{F}_{j}\right)^{-c}\left(d\mathbf{F}_{k}\right)=\frac{\Gamma_{m}[a_{k}]\Gamma_{m}[(m+1)/2]}{\Gamma_{m}[b_{k}]|\mathbf{A}_{k}|^{-a_{k}}}
  $$
\begin{equation}\label{prob1}
\hspace{1cm}\times \,     {}_{1}^{r_{k}}P_{1}\left(\frac{\prod_{j=1}^{k-1}|\mathbf{F}_{j}|^{g_{j}}\left(1+\displaystyle\sum_{j=1}^{k-1}\tr\mathbf{F}_{j}\right)^{-c-r_{k}}}{(c)_{r_{k}}^{-1}}:a_{k};b_{k}; -\mathbf{A}_{k}\right).
\end{equation}
Repeat a similar procedure to (\ref{prob1}) for the next integral on $\mathbf{0}<\mathbf{F}_{k-1}<\mathbf{A}_{k-1}$ by the corresponding splitting of the term indexed by $r_{k}$; thus the nested series ${}_{1}^{r_{k-1}}P_{1}\left[\cdot\right]$ is obtained. Explicitly, partial integration on $\mathbf{0}<\mathbf{F}_{k-1}<\mathbf{A}_{k-1},\mathbf{0}<\mathbf{F}_{k}<\mathbf{A}_{k}$ takes the form
$$
\frac{\prod_{i=k-1}^{k}\Gamma_{m}[a_{i}]\Gamma_{m}^{2}[(m+1)/2]}
{\prod_{i=k-1}^{k}\Gamma_{m}[b_{i}]}
  \prod_{i=k-1}^{k}|\mathbf{A}_{i}|^{a_{i}}\hspace{7cm}
$$
\begin{equation*}
\begin{tiny}
    \times \,     {}_{1}^{r_{k}}P_{1}\left(\frac{{}_{1}^{r_{k-1}}P_{1}\left(\frac{\prod_{j=1}^{k-2}|\mathbf{F}_{j}|^{g_{j}}\left(1+\displaystyle\sum_{j=1}^{k-2}\tr\mathbf{F}_{j}\right)^{-c-\sum_{i=k-1}^{k}r_{i}}}{\left(c+\sum_{i=k-1}^{k}r_{i}\right)_{r_{k-1}}^{-1}}:a_{k-1};b_{k-1}; -\mathbf{A}_{k-1}\right)}{(c)_{r_{k}}^{-1}}:\right.
\end{tiny}
\end{equation*}
$$
\hspace{10cm}\left.:a_{k};b_{k}; -\mathbf{A}_{k}\right).
$$
Induction on $r_{i}$, by a similar integration procedure leading to  (\ref{prob1}), provides the nested series ${}_{1}^{r_{k-2}}P_{1}\left[\cdot\right],\ldots,{}_{1}^{r_{2}}P_{1}\left[\cdot\right]$ for the probability on $\mathbf{0}<\mathbf{F}_{2}<\mathbf{A}_{2},\ldots,\mathbf{0}<\mathbf{F}_{k-1}<\mathbf{A}_{k-1},\mathbf{0}<\mathbf{F}_{k}<\mathbf{A}_{k}$. Finally, the last integral in $\mathbf{0}<\mathbf{F}_{1}<\mathbf{A}_{1}$ is given by:
$$\int_{\mathbf{0}<\mathbf{F}_{1}<\mathbf{A}_{1}} |\mathbf{F}_{1}|^{a_{1}}\left(1+\tr\mathbf{F}_{1}\right)^{-c-\sum_{i=2}^{k}r_{i}}(d\mathbf{F}_{1})=
$$
$$\frac{\Gamma_{m}[a_{1}]\Gamma_{m}[(m+1)/2]}{\Gamma_{m}[b_{1}]|\mathbf{A}_{1}|^{-a_{1}}}
  \quad{}_{1}^{r_{1}}P_{1}\left(\left(c+\sum_{i=2}^{k}r_{i}\right)_{r_{1}}:a_{1};b_{1}; \mathbf{A}_{1}\right),
  $$
  then the required probability (\ref{probbetaIIeq}) is obtained.
 \end{proof}

As usual in Wishart type probabilities, if $\lambda_{i}$ is the largest latent root of $\mathbf{F}_{i}$, then Theorem \ref{probbetaII} provides the distribution function  of $\lambda_{i}$, $P(\lambda_{i}<x)$, by taking $\mathbf{A}_{i}=x \mathbf{I}$ in (\ref{probbetaIIeq}).

We now focus on the \emph{multimatrix variate generalised Wishart
distribution} (\ref{mggwv}). This is the only distribution presented here that is not invariant under the spherical family. First note that (\ref{mggwv}) is indexed by $\int_{v>0}v^{Nm/2-1}h^{(r)}(v)dv$. It is a simple integral once the general derivatives are obtained. They can be found in \citet{cldggf:10} for the classical elliptical generators of Pearson, Kotz, Bessel, Jensen-Logistic and so on. For example, the Gaussian case of the multimatrix variate generalised Wishart distribution is obtained by taking the generator function $h(v)=(2\pi)^{-Nm/2}e^{-v/2}$, then 
$$
 \int_{v>0}v^{Nm/2-1}h^{(t)}(v)dv=\pi^{-Nm/2}\Gamma[Nm/2]\left(-\frac{1}{2}\right)^{t}.
$$
Excepting the extra summation, the procedure for Theorem \ref{probbetaII} can be reproduced next for probabilities on $m\times m$ positive definite matrices $\mathbf{W}_{i}$, $i=1,\ldots,k$.

\begin{theorem}\label{probwishart} Consider the hypothesis of Lemma \ref{ggW} and positive definite matrices $\mathbf{A}_{i} \in \Re^{m \times m}$, $i=1,\ldots,k$. Then we have the following probabilities:

1) Gaussian (independent) case: 
$$
P(\mathbf{0}<\mathbf{W}_{1}<\mathbf{A}_{1},\ldots,\mathbf{0}<\mathbf{W}_{k}<\mathbf{A}_{k},)=\hspace{7cm}
$$
\begin{equation}\label{probgenwishartgaussian}
\frac{\Gamma_{m}^{k}[(m+1)/2]\displaystyle\prod_{i=1}^{k}\left|\frac{1}{2}\mathbf{A}_{i}\right|^{n_{i}/2}}{\displaystyle\prod_{i=1}^{k}\Gamma_{m}[(n_{i}+m+1)/2]}
  \prod_{i=1}^{k}{}_{1}F_{1}\left[\frac{n_{i}}{2};\frac{n_{i}+m+1}{2};-\frac{1}{2}\mathbf{A}_{i}\right]
\end{equation}

2) General elliptical model, $k=1$ (\citet{clgfbn:16}): 
\begin{equation}\label{genM971}
P(\mathbf{0}<\mathbf{W}_{1}<\mathbf{A}_{1})=
\frac{\Gamma_{m}[\frac{m+1}{2}]|\mathbf{A}_{1}|^{\frac{n_{1}}{2}}}{\Gamma_{m}[\frac{n_{1}+m+1}{2}]\pi^{-\frac{n_{1}m}{2}}}
\sum_{t=0}^{\infty}\frac{h^{(t)}(0)}{t!}\ {}_{1}^{t}Q_{1}\left[\frac{n_{1}}{2};\frac{n_{1}+m+1}{2};\mathbf{A}_{1}\right]
\end{equation}
3) General elliptical model, for all $k> 1$: 
$$
P(\mathbf{0}<\mathbf{W}_{1}<\mathbf{A}_{1},\ldots,\mathbf{0}<\mathbf{W}_{k}<\mathbf{A}_{k},)=\hspace{7cm}
$$
$$
\frac{\pi^{(N-n_{0})m/2}\Gamma_{m}^{k}[(m+1)/2]\displaystyle\prod_{i=1}^{k}|\mathbf{A}_{i}|^{n_{i}/2}}{\displaystyle\prod_{i=1}^{k}\Gamma_{m}[(n_{i}+m+1)/2]}
\sum_{t=0}^{\infty}\frac{h^{(t)}(0)}{t!}\hspace{6cm}
$$
$$
  \times\, \sum_{r_{1}=0}^{t}\binom{t}{r_{1}}\,{}_{1}^{r_{1}}Q_{1}\left[\frac{n_{1}}{2};\frac{n_{1}+m+1}{2};\mathbf{A}_{1}\right]
 $$
$$
\times  \prod_{j=1}^{k-2}\left\{\sum_{r_{j+1}=0}^{t-\sum_{i=1}^{j}r_{i}}\binom{t-\sum_{i=1}^{j}r_{i}}{r_{j+1}}\,{}_{1}^{r_{j+1}}Q_{1}\left[\frac{n_{j+1}}{2};\frac{n_{j+1}+m+1}{2};\mathbf{A}_{j+1}\right]\right\}
$$
\begin{equation}\label{probgenwishartk}
\times\,{}_{1}^{r_{k}}Q_{1}\left[\frac{n_{k}}{2};\frac{n_{k}+m+1}{2};\mathbf{A}_{k}\right]
  \end{equation}
\end{theorem}
\begin{proof}
(\ref{probgenwishartgaussian}) is derived by using the generator function 
$h(v)=(2\pi)^{-Nm/2}e^{-v/2}$ in (\ref{mggwv}), and hence the integrals follows by $k$ independent application of the isotropic version of \citet[Cor. 7]{clgfbn:16} or \citet[Th. 9.7.1 ]{mh:05}. (\ref{genM971}) was derived in \citet[th.6]{clgfbn:16} as a generalization of the Gaussian case given in \citet[Th. 9.7.1 ]{mh:05}. 

Finally, (\ref{probgenwishartk}) follows by simple induction on $k$. It requires a recurrent use of the binomial theorem, the representation (\ref{genpartsum}) and the known integral 
$$
  \int_{\mathbf{0}<\mathbf{W}<\mathbf{A}}|\mathbf{W}|^{a -\frac{m+1}{2}}\tr^{r}(\mathbf{W})(d\mathbf{W}) \hspace{6cm}
$$
$$ 
  \hspace{4cm}
  = \frac{\Gamma_{m}[a]\Gamma_{m}[(m+1)/2]}{\Gamma_{m}[a+(m+1)/2]|\mathbf{A}|^{-a}}\, {}_{1}^{r}Q_{1}\left(a;a+\frac{m+1}{2};\mathbf{A}\right).
$$
\end{proof}

\begin{lemma}\label{ggbI}
In the setting of Lemma \ref{ggW} and the cone probabilities here derived, the only feasible lower probability is reached for $k=1$ and is given by:
$$P(\mathbf{0}<\mathbf{B}_{1}<\mathbf{A}_{1}<\mathbf{I})=\frac{\Gamma_{m}[n_{1}/2]\Gamma_{m}[(m+1))/2]}{\Gamma_{m}[(n_{1}+m+1))/2]}\hspace{6cm}$$
\begin{equation}\label{B1}
\times|\mathbf{A}_{1}|^{n_{1}/2}\sum_{t=0}^{\infty}\binom{n_{0}m/2-1}{t}{}_{1}^{t}Q_{1}\left[n_{1}/2;(n_{1}+m+1)/2;-\mathbf{A}_{1}\right].
  \end{equation}
For $k>1$ the probability $P(\mathbf{0}<\mathbf{B}_{i} < \mathbf{A}_{i} < \mathbf{I},\ldots,\mathbf{0} <\mathbf{B}_{k}<\mathbf{A}_{k} <\mathbf{I})$ turns in terms of invariant polynomials and it cannot be computable. 
\end{lemma}
\begin{proof}
Let $k=1$ and $\mathbf{A}_{1}$, such that $\mathbf{0}<\mathbf{B}_{1}<\mathbf{A}_{1}<\mathbf{I}$ and $\tr\mathbf{A}_{1}\leq 1$. The integrand in (\ref{b1}) simplifies to $|\mathbf{B}_{1}|^{(n_{1}-m-1)/2}\left(1-
  \tr \mathbf{B}_{1}\right)^{n_{0}m/2-1}$. Then, the result (\ref{B1}) follows after binomial theorem and application of 
$$
  \int_{\mathbf{0}<\mathbf{Y}<\mathbf{X}}|\mathbf{Y}|^{a-\frac{m+1}{2}}(\tr\mathbf{Y})^{r}(d\mathbf{Y}) \hspace{6cm}
$$ 
$$ \hspace{4cm}
  =\frac{\Gamma_{m}[a]\Gamma_{m}[(m+1)/2]}{\Gamma_{m}[a+(m+1)/2]|\mathbf{X}|^{-a}} \, {}_{1}^{r}Q_{1}\left(a;a+\frac{m+1}{2};\mathbf{X}\right).
$$ 
For $k > 1\,$ appear product of powers of traces arising invariant polynomials in the multiple integration, then the computation turns impossible by a similar Laplace Beltrami operator computation of zonal polynomials \citep{cl:16}.  
\end{proof}

We end this section in the context of real normed division algebras. We refer to \citet{b:02} and \citet{dggj:13} and the references therein for a complete exposition of the topic. For our purposes we index the four normed division algebras by the real dimension $\beta$, where $\beta=1$, stands for Real; $\beta=2$, for Complex; $\beta=4$, for Quaternionic; and $\beta=8$, for Octonion. Other notations for the algebras are given by $\alpha = 2/\beta$, see \citet{er:05}.

For the sequel some notations and definitions are required. For an understandable comparison with the real case in Section \ref{sec:2}, we provide a similar complete exposition for the real normed division algebras. Let ${\mathcal L}^{\beta}_{m,n}$ be the linear space of all $n \times m$ matrices of rank $m
\leq n$ over a real finite-dimensional normed division algebra $\mathfrak{F}$ with $m$ distinct positive singular values. Let $\mathfrak{F}^{n \times
m}$ be the set of all $n \times m$ matrices over $\mathfrak{F}$. The dimension of
$\mathfrak{F}^{n \times m}$ over $\Re$ is $\beta mn$. Let $\mathbf{A} \in \mathfrak{F}^{n
\times m}$, then $\mathbf{A}^{H} = \overline{\mathbf{A}}^{T}$ denotes the usual conjugate
transpose.

The set of matrices $\mathbf{H}_{1} \in \mathfrak{F}^{n \times m}$ such that
$\mathbf{H}_{1}^{H}\mathbf{H}_{1} = \mathbf{I}_{m}$ is a manifold denoted ${\mathcal
V}_{m,n}^{\beta}$, is termed the \emph{Stiefel manifold} ($\mathbf{H}_{1}$ is also known as
\emph{semi-orthogonal} ($\beta = 1$), \emph{semi-unitary} ($\beta = 2$), \emph{semi-symplectic}
($\beta = 4$) and \emph{semi-exceptional type} ($\beta = 8$) matrices, see \citet{dm:99}). The
dimension of $\mathcal{V}_{m,n}^{\beta}$ over $\Re$ is $[\beta mn - m(m-1)\beta/2 -m]$. In
particular, ${\mathcal V}_{m,m}^{\beta}$ with dimension over $\Re$, $[m(m+1)\beta/2 - m]$, is
the maximal compact subgroup $\mathfrak{U}^{\beta}(m)$ of ${\mathcal L}^{\beta}_{m,m}$ and
consists of all matrices $\mathbf{H} \in \mathfrak{F}^{m \times m}$ such that
$\mathbf{H}^{H}\mathbf{H} = \mathbf{I}_{m}$. Therefore, $\mathfrak{U}^{\beta}(m)$ is the
\emph{real orthogonal group} $\mathcal{O}(m)$ ($\beta = 1$), the \emph{unitary group}
$\mathcal{U}(m)$ ($\beta = 2$), \emph{compact symplectic group} $\mathcal{S}p(m)$ ($\beta = 4$)
or \emph{exceptional type matrices} $\mathcal{O}o(m)$ ($\beta = 8$), for $\mathfrak{F} = \Re$,
$\mathfrak{C}$, $\mathfrak{H}$ or $\mathfrak{O}$, respectively.

We denote by ${\mathfrak S}_{m}^{\beta}$ the real vector space of all $\mathbf{S} \in
\mathfrak{F}^{m \times m}$ such that $\mathbf{S} = \mathbf{S}^{H}$. Let
$\mathfrak{P}_{m}^{\beta}$ be the \emph{cone of positive definite matrices} $\mathbf{S} \in
\mathfrak{F}^{m \times m}$; then $\mathfrak{P}_{m}^{\beta}$ is an open subset of ${\mathfrak
S}_{m}^{\beta}$. Over $\Re$, ${\mathfrak S}_{m}^{\beta}$ consist of \emph{symmetric} matrices;
over $\mathfrak{C}$, \emph{Hermitian} matrices; over $\mathfrak{H}$, \emph{quaternionic
Hermitian} matrices (also termed \emph{self-dual matrices}) and over $\mathfrak{O}$,
\emph{octonionic Hermitian} matrices. Generically, the elements of $\mathfrak{S}_{m}^{\beta}$
are termed
 \textbf{Hermitian matrices}, irrespective of the nature of $\mathfrak{F}$. The
dimension of $\mathfrak{S}_{m}^{\beta}$ over $\Re$ is $[m(m-1)\beta+2m]/2$.

Let $\mathfrak{D}_{m}^{\beta}$ be the \emph{diagonal subgroup} of $\mathcal{L}_{m,m}^{\beta}$
consisting of all $\mathbf{D} \in \mathfrak{F}^{m \times m}$, $\mathbf{D} = \diag(d_{1},
\dots,d_{m})$.

For any matrix $\mathbf{X} \in \mathfrak{F}^{n \times m}$, $d\mathbf{X}$ denotes the\emph{
matrix of differentials} $(dx_{ij})$. Finally, we define the \emph{measure} or volume element
$(d\mathbf{X})$ when $\mathbf{X} \in \mathfrak{F}^{m \times n}, \mathfrak{S}_{m}^{\beta}$,
$\mathfrak{D}_{m}^{\beta}$ or $\mathcal{V}_{m,n}^{\beta}$.

If $\mathbf{X} \in \mathfrak{F}^{n \times m}$ then $(d\mathbf{X})$ (the Lebesgue measure in
$\mathfrak{F}^{n \times m}$) denotes the exterior product of the $\beta mn$ functionally
independent variables
$$
  (d\mathbf{X}) = \bigwedge_{i = 1}^{n}\bigwedge_{j = 1}^{m}dx_{ij} \quad \mbox{ where }
    \quad dx_{ij} = \bigwedge_{r = 1}^{\beta}dx_{ij}^{(r)}.
$$

If $\mathbf{S} \in \mathfrak{S}_{m}^{\beta}$ (or $\mathbf{S} \in \mathfrak{T}_{L}^{\beta}(m)$)
then $(d\mathbf{S})$ (the Lebesgue measure in $\mathfrak{S}_{m}^{\beta}$ or in
$\mathfrak{T}_{L}^{\beta}(m)$) denotes the exterior product of the $m(m+1)\beta/2$ functionally
independent variables (or denotes the exterior product of the $m(m-1)\beta/2 + m$ functionally
independent variables, if $s_{ii} \in \Re$ for all $i = 1, \dots, m$)
$$
  (d\mathbf{S}) = \left\{
                    \begin{array}{ll}
                      \displaystyle\bigwedge_{i \leq j}^{m}\bigwedge_{r = 1}^{\beta}ds_{ij}^{(r)}, &  \\
                      \displaystyle\bigwedge_{i=1}^{m} ds_{ii}\bigwedge_{i < j}^{m}\bigwedge_{r = 1}^{\beta}ds_{ij}^{(r)}, &
                       \hbox{if } s_{ii} \in \Re.
                    \end{array}
                  \right.
$$
The context generally establishes the conditions on the elements of $\mathbf{S}$, that is, if
$s_{ij} \in \Re$, $\in \mathfrak{C}$, $\in \mathfrak{H}$ or $ \in \mathfrak{O}$. It is
considered that
$$
  (d\mathbf{S}) = \bigwedge_{i \leq j}^{m}\bigwedge_{r = 1}^{\beta}ds_{ij}^{(r)}
   \equiv \bigwedge_{i=1}^{m} ds_{ii}\bigwedge_{i < j}^{m}\bigwedge_{r =
1}^{\beta}ds_{ij}^{(r)}.
$$
Observe, too, that for the Lebesgue measure $(d\mathbf{S})$ defined thus, it is required that
$\mathbf{S} \in \mathfrak{P}_{m}^{\beta}$, that is, $\mathbf{S}$ must be a non singular
Hermitian matrix (Hermitian positive definite matrix).

If $\mathbf{\Lambda} \in \mathfrak{D}_{m}^{\beta}$ then $(d\mathbf{\Lambda})$ (the Legesgue
measure in $\mathfrak{D}_{m}^{\beta}$) denotes the exterior product of the $\beta m$
functionally independent variables
$$
  (d\mathbf{\Lambda}) = \bigwedge_{i = 1}^{n}\bigwedge_{r = 1}^{\beta}d\lambda_{i}^{(r)}.
$$

If $\mathbf{H}_{1} \in \mathcal{V}_{m,n}^{\beta}$ then
$$
  (\mathbf{H}^{H}_{1}d\mathbf{H}_{1}) = \bigwedge_{i=1}^{n} \bigwedge_{j =i+1}^{m}
  \mathbf{h}_{j}^{H}d\mathbf{h}_{i}.
$$
where $\mathbf{H} = (\mathbf{H}_{1}|\mathbf{H}_{2}) = (\mathbf{h}_{1}, \dots,
\mathbf{h}_{m}|\mathbf{h}_{m+1}, \dots, \mathbf{h}_{n}) \in \mathfrak{U}^{\beta}(m)$. It can be
proved that this differential form does not depend on the choice of the $\mathbf{H}_{2}$
matrix. When $m = 1$; $\mathcal{V}^{\beta}_{1,n}$ defines the unit sphere in
$\mathfrak{F}^{n}$. This is, of course, an $(n-1)\beta$- dimensional surface in
$\mathfrak{F}^{n}$. When $m = n$ and denoting $\mathbf{H}_{1}$ by $\mathbf{H}$,
$(\mathbf{H}^{H}d\mathbf{H})$ is termed the \emph{Haar measure} on $\mathfrak{U}^{\beta}(m)$.

The surface area or volume of the Stiefel manifold $\mathcal{V}^{\beta}_{m,n}$ is
\begin{equation}\label{vol}
    \Vol(\mathcal{V}^{\beta}_{m,n}) = \int_{\mathbf{H}_{1} \in
  \mathcal{V}^{\beta}_{m,n}} (\mathbf{H}^{H}_{1}d\mathbf{H}_{1}) =
  \frac{2^{m}\pi^{mn\beta/2}}{\Gamma^{\beta}_{m}[n\beta/2]},
\end{equation}
where $\Gamma^{\beta}_{m}[a]$ denotes the multivariate Gamma function for the space
$\mathfrak{S}_{m}^{\beta}$, and is defined by
\begin{eqnarray*}
  \Gamma_{m}^{\beta}[a] &=& \displaystyle\int_{\mathbf{A} \in \mathfrak{P}_{m}^{\beta}}
  \etr\{-\mathbf{A}\} |\mathbf{A}|^{a-(m-1)\beta/2 - 1}(d\mathbf{A}) \\
&=& \pi^{m(m-1)\beta/4}\displaystyle\prod_{i=1}^{m} \Gamma[a-(i-1)\beta/2],
\end{eqnarray*}
where $\etr(\cdot) = \exp(\tr(\cdot))$, $|\cdot|$ denotes the determinant and $\re(a)> (m-1)\beta/2$, see \citet{gr:87}. If $\mathbf{A}\in \mathcal{L}_{m,n}^{\beta}$ then by $\vec (\mathbf{A})$ we mean the $mn \times 1$ vector formed by stacking the columns of
$\mathbf{A}$ under each other.

Now, generalized statistical theory of shape has been developed by the authors in a number of settings: SVD, polar, PseudoWishart, QR, affine, forms, Eulerian, etc.. In particular, the real configuration or affine density was set in the addressed families of elliptically countored distributions in \citet{cldggf:10} as a revision and generalization of the Gaussian case given by \citet{gm:93}. Then \citet{dgcl:16} proposed the affine shape theory under the general approach for real normed division algebras. The configuration or affine shape filters are interested in removing geometrical information about translation, scaling, rotation, reflection and/or uniform shearing of random objects summarized by matrices in the addressed four real normed divison algebras. Explicitly, two figures $\mathbf{X}\in {\mathcal L}^{\beta}_{K,N}$ have the same affine shape if $\mathbf{X}_{1}=\mathbf{X}\mathbf{E}+\mathbf{1}_{N}e^{*}$, for some translation $e\in{\mathcal L}^{\beta}_{1,K}$ and $\mathbf{E}\in {\mathcal L}^{\beta}_{K,K}$. Then the $(N-1)\times K$ configuration matrix $\mathbf{U}=(\mathbf{I}|(\mathbf{Y}_{2}\mathbf{Y}_{1}^{-1})')'$ compressing the meaningful geometrical information of the original $N\times K$ matrix $\mathbf{Z}$, is obtained in the sequence of filtering geometrical stages $\mathbf{L}\mathbf{Z}=\mathbf{Y}=\mathbf{U}\mathbf{E}$. Here $\mathbf{Y}=(\mathbf{Y}_{1}'|\mathbf{Y}_{2}')'$ and $\mathbf{L}$ is a subHelmert matrix, see \citet{dgcl:16}, for details. 
Now, a similar definition to (\ref{elliptical}) emerges for real normed division algebras. We say that $\mathbf{X}\in {\mathcal L}^{\beta}_{m,n}$ has the following multivariate elliptically contoured distribution for real normed division algebras respect to the Lebesgue measure:
\begin{equation}\label{RNDAelliptical}
 F_{_{\mathbf{X}}} (\mathbf{X}) = |\mathbf{\Sigma}|^{-\beta m/2}|\mathbf{\Theta}|^{-\beta n/2}
  h\left\{\tr\left[\mathbf{\Theta}^{-1}(\mathbf{X}-\boldsymbol{\mu})^{*}\mathbf{\Sigma}^{-1}(\mathbf{X}-
  \boldsymbol{\mu})\right]\right\},
\end{equation}
where $\boldsymbol{\mu}\in {\mathcal L}^{\beta}_{m,n}$, $\boldsymbol{\Sigma}\in  \mathfrak{P}_{n}^{\beta}$, $\boldsymbol{\Theta}\in \mathfrak{P}_{n}^{\beta}$, and generator function $h\mbox{: } \mathfrak{F} \to [0,\infty)$, satisfies $\int_{\mathfrak{P}_{1}^{\beta}} u^{\beta nm-1}h(u^{2})du < \infty$. This fact will be denoted by $\mathbf{X}\sim \mathcal{E}_{n \times m}^{\beta}(\boldsymbol{\mu}
,\boldsymbol{\Sigma}\otimes\boldsymbol{\Theta}; h)$ Finally, for a convergent Taylor series of $h(\cdot)$, if $\mathbf{Y}\sim \mathcal{E}_{N-1 \times K}^{\beta}(\boldsymbol{\mu}\boldsymbol{\Theta}^{-1/2}
,\boldsymbol{\Sigma}\otimes\mathbf{I}_{K}, h)$, the affine shape density of $\mathbf{U}$ is given by
\begin{equation}\label{config}
g(\mathbf{U},t,r)\gamma_{t,r},
  \end{equation}
with
$$g(\mathbf{U},t,r)=\frac{\pi^{\beta K^{2}/2}\Gamma_{K}^{\beta}[\beta(N-1)/2]}{\Gamma_{K}^{\beta}[\beta K/2]|\boldsymbol{\Sigma}|^{\beta K/2}|\mathbf{U}^{*}\boldsymbol{\Sigma}^{-1}\mathbf{U}|^{\beta(N-1)/2}}\sum_{t=0}^{\infty}\frac{1}{t!\Gamma[K(N-1)/2+t]}\hspace{6cm}$$
\begin{equation}\label{Affine2}
\times\sum_{r=0}^{\infty}\frac{\tr^{r}\boldsymbol{\Omega}}{r!}\sum_{\tau}\frac{\left(\beta(N-1)/2\right)_{\tau}^{\beta}}{\left(\beta K/2\right)_{\tau}^{\beta}}C_{\tau}^{\beta}\left(
\mathbf{U}^{*}\Omega\boldsymbol{\Sigma}^{-1}\mathbf{U}
(\mathbf{U}^{*}\boldsymbol{\Sigma}^{-1}\mathbf{U})^{-1}\right),
  \end{equation}
and 
\begin{equation}\label{intAffine}
\gamma_{t,r}=\int_{\mathfrak{v}_{1}^{\beta}} v^{\beta K(N-1)/2+t-1}h^{(2t+r)}(v)dv < \infty,
  \end{equation}
$\boldsymbol{\Sigma}=\mathbf{L}\boldsymbol{\Sigma}_{\mathbf{X}}\mathbf{L}^{*}$, $\boldsymbol{\mu}=\mathbf{L}\boldsymbol{\mu}_{\mathbf{X}}$,  $\boldsymbol{\Omega}=\boldsymbol{\Sigma}^{-1}\boldsymbol{\mu}\boldsymbol{\Theta}\boldsymbol{\mu}^{*}$.

Now, in the computational context of this Section, the distribution (\ref{config}) is completely feasible; \citet{cldggf:10} and \citet{dgcl:16} have provided a number of applications. However, the density requires the integral (\ref{intAffine}) which is set in terms of the $2t+r$ general derivative of the kernel function $h(\cdot)$. Inspired by Theorem \ref{probwishart}, we can avoid integration and non null evaluation of the derivative in (\ref{intAffine}) by a simple computation of $h^{(2t+r)}(0)$. Next result provides the solution for the three models given by \citet{dggj:13}.

\begin{theorem}\label{th:newkernelintegral}
The affine shape density of $\mathbf{U}$ for Gaussian, Pearson type VII and II models is given by
\begin{equation}\label{config2}
g(\mathbf{U},t,r)\zeta_{t,r},
  \end{equation}
    where:
    \begin{itemize}
        \item Gaussian or Hermitian:
        $$
         \zeta_{t,r}=(2\beta^{-1})^{t+\beta K(N-1)/2}\Gamma[t+\beta K(N-1)/2]h^{(2t+r)}(0),
        $$
        with $h(y)=(2\pi\beta^{-1})^{-\beta K(N-1)/2}\exp(-\beta y/2)$.
        \item T type I or Pearson type VII: 
        $$
          \zeta_{t,r}=\frac{\beta^{-t-\beta K(N-1)/2}\Gamma[t+\beta K(N-1)/2]\Gamma[t+r+\beta \nu/2]}{\Gamma[2t+r+\beta(K(N-1)+\nu)/2]}h^{(2t+r)}(0),
        $$
        with 
        $$
          h(y)=\frac{\Gamma_{1}^{\beta}[\beta(K(N-1)+\nu)/2]}{(\pi\beta^{-1})^{\beta K(N-1)/2} \Gamma_{1}^{\beta}[\beta\nu/2]} \left(1+\beta y\right)^{-\beta(K(N-1)+\nu)/2}.
        $$
        \item Gegenbauer type I or Pearson type II: 
        $$
          \zeta_{t,r}=\frac{\Gamma[t+\beta K(N-1)/2]\Gamma[r+t-\beta K(N-1)/2-\beta q]}{(-\beta)^{t+\beta K(N-1)/2}\Gamma[2t+r-\beta q]}h^{(2t+r)}(0),
        $$ 
         with 
        $$
          h(y)=\frac{\Gamma_{1}^{\beta}[\beta K(N-1)/2+\beta q+1]}{(\pi\beta^{-1})^{\beta K(N-1)/2}\Gamma_{1}^{\beta}[\beta q+1]}\left(1-\beta y\right)^{\beta q}.
        $$
    \end{itemize}
\end{theorem}

Finally, \citet{dgcl:24b} established a paralellism between the multimatrix variate distributions, studied here, and the multimatricvariate distributions. They share the simplicity that we have addressed for computation of multiple probabilities. Instead of traces, the corresponding determinants can be expanded in terms of zonal polynomials and then a similar computation can be performed for the multimatricvariate distributions termed as Pearson VII, Pearson type II and beta type II in \citet{dgcl:22}.

\section{Application in dynamic molecular docking in SARS-CoV-2}\label{sec:4}

Recently, \citet{dgcl:24a} applied the multimatrix variate distributions in a problem of molecular docking, by finding a new cavity of 241 atoms inside the SARS-CoV-2 main protease for placing the inhibitor N3 of 21 atoms. The algorithm of searching was based on a theorem provided by \citet{rbvc:22} and the main protease of 2387 atoms and the ligand emerged from \citet{jdx:20} and \citet{PDB:20}. A multimatrix setting models dependent sample experiments by a realistic estimation based on non independent likelihoods. In this situation, 56 movements of the rigid ligand N3 was recorded into the new pocket meanwhile the equilibrium is reached by a decreasing Lennard-Jones potential type $6-12$ and $6-10$. Finally, the dependent sample joint distribution was estimated as: 
\begin{equation*}
     \frac{\Gamma_{1}[(a_{0}+ka)m]}{\Gamma_{1}[a_{0}m]\Gamma_{m}^{k}[a]}
  \prod_{i=1}^{k}|\mathbf{F}_{i}|^{a-\frac{m+1}{2}}
\left(1+\displaystyle\sum_{i=1}^{k}\tr\mathbf{F}_{i}\right)^{-(a_{0}+ka)m},
\end{equation*}
where $a_{0}=0.34397, a=0.19735$, $m=3, k=56$, $\mathbf{F}_{i}=\mathbf{T}_{i}'\mathbf{T}_{i}, i=1,\ldots,k$, for $21\times 3$ matrices $\mathbf{T}_{i}$ registering the spatial coordinates of the ligand inside the protease (see \citet{dgcl:24a} for details). We also highlight that the estimation was completed under the full invariant family of spherical distributions. A fact that eliminates the complex problem of previous knowledge and/or fitting of the hidden law. The dependent joint distribution is also crucial here because the molecular docking experiment demands a time dependent calibration which can not reached by the classical (probabilistic independent) likelihood estimation. 

The addressed example of \citet{dgcl:24a} showed that the joint distribution functions based on spherical multimatrix distributions are easily computable and applied to real data. We now advance into a more complex example involving series of zonal polynomials.

The theorem in \citet[Th. Sec. 3.]{rbvc:22} holds only for rigid molecular docking, it means that the ligand is optimally placed in the pocket by rotations and without distortions. However, molecular biology states that the interaction of one inhibitor inside an active site is dynamic, forcing that the ligand could change its shape by suitable deformations. One way of emulating that molecular dynamic consists of computing the probability that the ligand changes its shape according a desirable coupling structure in the protein. The probability can be modeled by the latent roots of certain fixed positive definite matrix representing the active site near the ligand. For the sake of a simple mathematical illustration, and without any biological and expert study, consider the nearest 21 atoms to the optimized ligand (time 56). Figure \ref{fig2} shows both groups of atoms. Let $\mathbf{A}$ the corresponding $3\times 3$ symmetrized matrix of the nearest atoms in the protein. 
\begin{figure}[h]
    \centering
    \includegraphics[width=4cm]{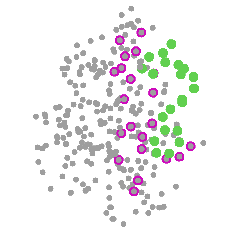}
    \caption{Initial setting of the lowest Lennard-Jones potential rigid ligand (last locus 56, green) and the first 21 nearest atoms (magenta) in the selected cavity on SARS-CoV-2 main protease (grey). The symmetrized optimal ligand is apart 0.24884 in positive definite probability from the symmetrized surrounded neighborhood. It is expected to improve the affinity by deforming the ligand in such way that the $P(\mathbf{0}<\mathbf{F}<\mathbf{A})$ is greater, then the latent roots information about the ligand shape is near to the geometry of the symmetrized target.}
    \label{fig2}
\end{figure}
The latent roots of $\mathbf{A}$ are $A_{1}=0.5864, A_{2}=0.2351, A_{3}=0.1785$, a fact that reflects an embracing neighborhood cavity. Meanwhile, the more stable ligand (highly flat) is represented by the symmetric matrix with latent roots $L_{1}= 0.8663, L_{2}=0.0991, L_{3}=0.0346$.
Thus, by Theorem \ref{probbetaII}, the probability that a symmetrized ligand $\mathbf{F}$ is near to the symmetrized neighborhood $\mathbf{A}$ in terms of definite positivity is given by:
$$P(\mathbf{0}<\mathbf{F}<\mathbf{A})=\hspace{10cm}
  $$
  $$
  \frac{\Gamma_{1}[(a_{0}+ka)m]\Gamma_{m}[(m+1)/2]}{\Gamma_{1}[a_{0}m]\Gamma_{m}[a+(m+1)/2]}|\mathbf{A}_{1}|^{a}\,{}_{1}^{r}P_{1}\left[\left((a_{0}+ka)m\right)_{r}:a;a+(m+1);-\mathbf{A}\right]
  $$
In this case $m=3$, thus the ${}_{1}^{r}P_{1}[\cdot]$ can be written as $$
\sum_{r=0}^{\infty}\frac{\left((a_{0}+ka)m\right)_{r}\,{}_{1}^{r}Q_{1}\left[\left((a_{0}+ka)m\right)_{r}:a;a+(m+1);-\mathbf{A}\right]}{r!},
$$ and ${}_{1}^{r}Q_{1}\left[\cdot\right]$ is just a summation over partitions of 3 parts which can be easily write down, then the computation of the zonal polynomials of only 3 latent roots can be computed by using \citet{gr:78} combined with the recurrence method of \citet{j:68}; or directly by the referred modification of the algorithms for hypergeometric series given by \citet{ke:06}. 
For tracing a path in the probability model, we just sweep the three latent roots $\delta_{1}, \delta_{2}, \delta_{3}$ of a variable $\mathbf{A}$, as a decreasing distance from $L_{1}, L_{2}, L_{2}$, respectively. Explicitly, we take $n=1000$ triples extracted from the sequences $\lambda_{i,j}$ from $0$ to $L_{i}-A_{i}$ by $(L_{i}-A_{i})/n$, $i=1,2,3; j=1,\ldots,n$. Then the upper bound definite matrix $\mathbf{A}_{j}, j=1,\ldots,n$, representing the bent or deformed ligand, have the latent roots $A_{1,j}=L_{1}-\lambda_{1,j}$, $A_{2,j}=L_{2}+\lambda_{2,j}$, $A_{3,j}=L_{3}+\lambda_{3,j},$ respectively. Finally, the 1000 results are depicted in Figure \ref{fig3}, they are increasing probabilities from $\mathbf{A_{1}}=\mathbf{F_{56}}$, which is the rigid optimized ligand in the last time (56) of the dependent sample, to $\mathbf{A_{1000}}$ with the specified latent roots $A_{1}, A_{2}, A_{3}$ corresponding to the target  neighborhood in the protein.  
The first probability is $P(\mathbf{0}<\mathbf{F}<\mathbf{A}_{1})=0.24884$ which is less than the associated probability $P(\mathbf{0}<\mathbf{F}<\mathbf{A}_{1000})=0.37276$. 

\begin{figure}[htp]
    \centering
    \includegraphics[width=6cm]{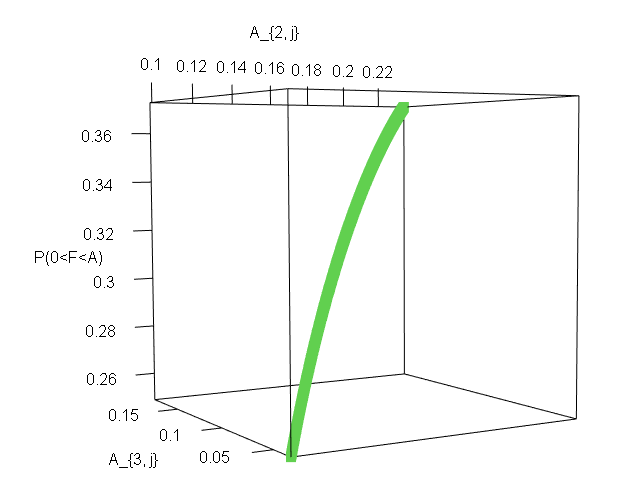}
    \caption{The probability path of 1000 transitions from the symmetrized optmized rigid and flat ligand to a positive definite close distance of symmetrized 21 surrounding atoms in a cavity found on SARS-CoV-2 main protease. The trayectory starts in $P(\mathbf{0}<\mathbf{F}<\mathbf{A}_{1})=0.24884$, where the ligand is flat and non-deformed, then by suitable increments of its second and third symmetrized latent roots, and decrement of the first, the probability is increased until $P(\mathbf{0}<\mathbf{F}<\mathbf{A}_{1000})=0.37276$. In the last stage the symetrized ligand reaches approximate near latent roots of the symmetrized neigborhood with greater probability, a fact that describes a plausible molecular docking by a non flat rigid coupling.}
    \label{fig3}
\end{figure}

Then a distorted symmetrized ligand getting more similar in latent roots to the symmetrized target in the pocket appear with greater probability, a fact that can be studied in future as a punctuation of an effective molecular docking under non rigid movements and a plausible calibration with the Lennard-Jones potential. 

Finally, the above technique can be extended to two o more probabilities on Theorem \ref{probbetaII}, it just needs to compute ${}_{1}^{r}Q_{i}\left[\cdot\right]$ as weights of the corresponding coefficient of the $i-th$ summation.   

\section{Conclusions}
This work revises the problem of probability computations on cones in matrix variate distributions and applies the discussion in a several situations. In particular, the probabilities of multimatrix variate distributions are set in terms of computable series of zonal polynomials. Some of the distributions here applied are invariant under the complete family of elliptically countered distributions, they include the termed Pearson II, Pearson VII. The non invariant case is also considered in the generalised multimatrix variate Wishart distributions. This case promotes a simplification of a classical kernel integral in elliptically contoured distributions which was applied in statistical shape affine distributions under real normed division algebras. Finally, the study of probabilities on cones can be applied in some meaningful situations by understanding the positive definiteness probability in a context of dynamic molecular docking in SARS-CoV-2. 



\end{document}